\documentclass[12pt,reqno]{amsart}
\usepackage{amsmath}
\usepackage{amssymb}
\usepackage{amsfonts, stmaryrd}
\usepackage{a4wide}
\usepackage{amssymb,latexsym}
\usepackage{enumerate}
\usepackage{color}
\newcommand{\caap}{\mathrm{cap}}

\newcommand{\TT}{\mathbb{T}}
\newcommand{\DD}{\mathbb{D}}

\DeclareMathOperator{\Hol}{Hol}
\newtheorem{thm}{Theorem}[section]
\newtheorem{lem}[thm]{Lemma}
\newtheorem{cor}[thm]{Corollary}
\newtheorem{rem}[thm]{Remark}
\newtheorem{prop}[thm]{Proposition}

\newcommand{\cD}{{\mathcal{D}}}
\renewcommand{\Re}{\mathrm{Re\;}}
\renewcommand{\Im}{\mathrm{Im\;}}
\newcommand{\T}{\mathbb{T}}

\newcommand{\Z}{\mathbb{Z}}

\newcommand{\D}{\mathbb{D}}

\numberwithin{equation}{section}

\begin{document}

\title[Zero sets]{On zero sets in the Dirichlet space}

\author{Karim Kellay}
\address{CMI, LATP,
         Universit\'e de Provence,
         39, Rue F. Joliot-Curie,
         13453 Marseille Cedex 13, France}
\email{kellay@cmi.univ-mrs.fr}

\author{Javad Mashreghi}
\address{D\'epartement de math\'ematiques et de statistique,
         Universit\'e Laval,
         Qu\'ebec, QC,
         Canada G1K 7P4.}
\email{javad.mashreghi@mat.ulaval.ca}

\date{}

\begin{abstract}
We study the zeros sets of functions in the Dirichlet space. Using Carleson formula for Dirichlet integral, we obtain some new families of zero sets. We also show that any closed subset of $E \subset \TT$ with logarithmic capacity zero is the accumulation points of the zeros of a function in the Dirichlet space.  The zeros satisfy a growth restriction which depends on $E$.
\end{abstract}



\thanks{This work was supported by NSERC (Canada), NRS and ANRDynop (France).}

\maketitle

\section{Introduction}

The {\em Dirichlet space}  $\cD$ consists of all analytic functions $f =\sum_{n\geq0} a_n z^n$, defined on the open unit disk $\DD$, such that
\[
\cD(f) := \int_\D |f'(z)|^2 \, dA(z) = \sum_{n=1}^\infty n \, |a_n|^2 <\infty,
\]
where $dA(z)=dxdy/\pi$ istands for the normalized area measure in $\DD$ .  A sequence $(z_n)_{n \geq 1}$ in $\D$ is called a {\em zero set} for $\cD$ provided that there is a function
$f \in \DD$, $f \not \equiv 0$, such that $f(z_n) = 0$, $n \geq 1$. Since $\cD$ is contained in the Hardy space $H^2(\D)$, the zero set $(z_n)_{n \geq 1}$ must satisfy the Blaschke conditions
\[
\sum_{n=1}^{\infty} (1-|z_n|) < \infty.
\]
However, the complete characterization of zero sets of $\cD$ is still an {\em open question}. For a short history of this topic, see \cite{ms-4}. Very briefly, let us mention that Shapiro--Shields \cite{SS62} improved an earlier result of Carleson \cite{Car52-2} by showing that
\begin{equation} \label{E:suuf-z-ss}
\sum_{n=1}^{\infty} \frac{1}{|\log(1-|z_n|)|}  < \infty
\end{equation}
is enough to ensure that $(z_n)_{n \geq 1}$ is a zero set for $\cD$. On the other hand, Nagel--Rudin--Shapiro \cite{NRS82} and Richter--Ross--Sundberg \cite{RRS04} showed that for any sequence $(r_n)_{n \geq 1}$ which do not do not fulfill \eqref{E:suuf-z-ss}, one can choose $(\theta_n)_{n \geq 1}$ such that $(r_ne^{i\theta_n})_{n \geq 1}$ is a uniqueness set for $\cD$. This line of research has been continued in  \cite{Bog96, Cau69,  ms-4, ms-5, PP}.

In this note, we use  Carleson's formula for the Dirichlet integral \cite{Car60} to obtain some new families of the zero sets which are not given by the classical theorems mentioned above. We highlight two main results below.

\begin{thm}\label{module}
Let $\omega:[0,2] \longrightarrow [0,\infty)$ be a continuous increasing function such that $\omega(0)=0$  and
\begin{equation}\label{cond_regularite}
\int_{\delta}^{2}\frac{\omega(t)}{t^2} \, dt = O\left( 1+\frac{\omega(\delta)}{\delta} \right), \qquad \delta \longrightarrow 0.
\end{equation}
 Let $E \subset \TT$ and suppose that there exists a function $f \in \cD$, $f\neq 0$, such that
\begin{equation}\label{cond-regulariteppp-0}
|f(\zeta)|^2\leq \omega \big( d(\zeta,E) \big), \qquad \text{a.e. on }  \TT.
\end{equation}
Then every Blaschke sequence $(z_n)_{n\geq 1}$ satisfying
\[
\sum_{n=1}^\infty \, \omega \big( 2d(z_n,E) \big) < \infty
\]
is a zero set for $\cD$.
\end{thm}

The function $\omega(t) = t$ is close to the borderline for the hypothesis of Theorem \ref{module}. The most demanding condition in this theorem is the existence of a function $f \in \cD$ which satisfies \eqref{cond-regulariteppp-0}. However, under some mild extra conditions, such a function can be explicitly constructed. In Section \ref{S:examp}, for {\em Carleson sets} or for {\em sets of logarithmic capacity zero} we construct a function $f$ which fulfills \eqref{cond_regularite} and \eqref{cond-regulariteppp-0}. This is analogue to the result of that Pau-- Pel\'aez \cite[Theorem 1]{PP} about the zeros of function in the Dirichlet type space
$$\cD_\alpha:=\Big\{f\inÊ\Hol(\DD): \int_\DD|f'(z)|^2(1-|z|^2)^\alpha dA(z)<\infty\Big\}$$
for $0<\alpha<1$, in which more conditions are required on the zero set.

In order to state our second result, we introduce the notion of logarithmic capacity.
The {\em energy} of a Borel probability measure $\mu$ on $\TT$ is defined by
\[
I(\mu) = \int\int \log\frac{1}{|\zeta - \xi|} \, d\mu(\zeta) \, d\mu(\xi).
\]
A simple calculation shows that
\[
I(\mu)=\sum_{n=1}^\infty \frac{|\widehat{\mu}(n)|^2}{n},
\]
where $\big( \widehat{\mu}(n) \big)_{n \in \Z}$ is the sequence of Fourier coefficients of $\mu$. Given a Borel subset $E$ of $\TT$, we denote by $\mathcal{P}(E)$ the set of all probability measures supported on a compact subset of $E$. Then we define the {\em logarithmic capacity} of $E$ by
\[
\caap (E) = 1/\inf\{I(\mu) \text{ : } \mu \in \mathcal{P}(E) \}.
\]
According to a well-known result of Beurling \cite{C}, for each function $f\in \cD$, the radial limits $f^*(\zeta)=\lim_{r \to 1-}f(r\zeta) $ exist q.e. on $\TT$, that is
\[
\caap \big(\, \{\zeta\in \TT\text{ : } f^*(\zeta) \text{ does not  exist}\} \,\big)=0.
\]

If $E$ is any closed subset of $\TT$, we can take $(r_n)_{n \geq 1}$ to be any positive sequence which satisfies the Shapiro--Shields condition \eqref{E:suuf-z-ss}, and then adjust the arguments $(\theta_n)_{n \geq 1}$ such that $(r_ne^{i\theta_n})_{n \geq 1}$ accumulates precisely at all points of $E$. But such a sequence converges very rapidly to the boundary. However, for a closed set of logarithmic capacity zero,  using Theorem \ref{module}, we can construct a sequence which fulfils the above property and at the same time grows much slower to its accumulation points on $\T$. See also Remark \ref{rem}.

For any subset $E$ of  $\TT$, we define
\[
E_t=\{\zeta\in \TT : d(\zeta,E)\leq t\}, \qquad t>0.
\]

\begin{thm}\label{capacity}
Let $E$ be a closed subset of  $\TT$ with logarithmic capacity zero. Let $\psi$ be any positive continuous decreasing function  such that
$\psi(x) \longrightarrow 0$  as $x \longrightarrow +\infty$ and $\int^{+\infty }\psi(x)\, x\, dx<\infty$, and then define
\[
\omega(t) = \exp\big(-e^{ \psi^{-1}(\caap(E_t))}\big).
\]
Let $(r_n)_{n \geq 1}$ be any positive sequence which satisfies the growth restriction
\begin{equation}\label{blas}
\sum_{n=1}^{\infty} (1-r_{n}) \int_{2(1-r_{n})}\frac{\omega(t)}{t^2} \, dt<\infty.
\end{equation}
Then we can choose the arguments $(\theta_n)_{n \geq 1}$ such that $(r_ne^{i\theta_n})_{n \geq 1}$ is a zero set for $\cD$, and the set of its accumulation points on $\TT$ is precisely equal to $E$.
\end{thm}

In Section 4, we show that if $(z_n)_{n \geq 1}$ is a zero sequence satisfying the Shapiro--Shields condition  \eqref{E:suuf-z-ss}, then  the angular derivative of the Blaschke product formed with $(z_n)_{n \geq 1}$ exists--q.e.  On the other direction, using Theorem \ref{module}, we give a zero sequence $(z_n)_{n \geq 1}$ for $\cD$ such that the angular derivative of the corresponding Blaschke product exists precisely on $\TT\backslash E$, where $E$ is a set with $\caap(E)>0$.

\section{Proofs}
For the proofs of above theorems, we need the following key lemma.

\begin{lem}\label{T:zeroset-phi}
Let $\omega:[0,2] \longrightarrow [0,\infty)$ be a continuous increasing function such that $\omega(0)=0$. Let $E \subset \TT$ and suppose that there exists a function $f \in \cD$, $f\neq 0$, such that
\[
|f(\zeta)|^2\leq \omega \big( d(\zeta,E) \big), \qquad \text{a.e. on }  \TT.
\]
Then every Blaschke sequence $(z_n)_{n\geq 1}$ satisfying
\[
\sum_{n=1}^\infty \left(\, \omega \big( 2d(z_n,E) \big) + (1-|z_n|)\int_{2d(z_n,E)}\frac{\omega(t)}{t^2} \, dt \,\right) < \infty
\]
is a zero set for $\cD$.
\end{lem}

\begin{proof}Fix $z \in \D$. Let
\[
\Gamma= \{\zeta \in \T :  d(\zeta, E)\geq d(z,E)\},
\]
and consider the decomposition
\[
\Gamma = \bigcup_{n=0}^{\infty} \Gamma_z^n,
\]
where
\[
\Gamma_z^n = \left\{\, \zeta \in \T:  2^n d(z,E)\leq d(\zeta,E) <  2^{n+1}d(z,E) \,\right\}.
\]
On $\Gamma_z^0$, we have
\begin{eqnarray*}
\int_{\Gamma_z^0} \frac{|f(\zeta)|^2}{|\zeta-z|^2} \,\, |d\zeta|
&\leq& \int_{\Gamma_z^0 } \frac{\omega(\, d(\zeta,E))}{|\zeta-z|^2} \,\, |d\zeta| \\
&\leq&  \omega(2d(z,E)) \, \int_\T\frac{|d\zeta|}{|\zeta-z|^2} \\
&=& 2\pi \frac{\omega(2d(z,E))}{1-|z|}.
\end{eqnarray*}

Let $\zeta_0$ be the closest point of $E$ to $z$. Then, for each $\zeta \in \Gamma_z^n$, $n \geq 1$, we have
\begin{eqnarray*}
|\zeta-z| &\geq& |\zeta-\zeta_0|-|\zeta_0-z|\\
&\geq& d(\zeta,E)-d(z,E)\\
&\geq& 2^{n-1} \, d(z,E).
\end{eqnarray*}
Hence,  $\Gamma_n^z\subset \{\zeta\in\T : |\zeta-z|\geq 2^{n-1}d(z,E)\}$, and  we obtain
\begin{eqnarray*}
\int_{\Gamma_z^n } \frac{|f(\zeta)|^2}{|\zeta-z|^2} \,\, |d\zeta|
&\leq& \int_{\Gamma_z^n }
\frac{\omega(d(\zeta,E))}{|\zeta-z|^2} \,\, |d\zeta|\\
&\leq&  {\omega( 2^{n+1}d(z,E))}  \int_{\Gamma_z^n }
\frac{|d\zeta|}{|\zeta-z|^2}  \\
&\leq&  {\omega( 2^{n+1}d(z,E))}  \int\limits_{ \{\zeta\; : |\zeta-z|\geq 2^{n-1}d(z,E)\}  }
\frac{|d\zeta|}{|\zeta-z|^2} \\
&\leq& {\omega( 2^{n+1}d(z,E))}  \int_{ 2^{n-1}d(z,E)  }^{\infty}
\frac{dx}{x^2} \\
&=& 4\;\frac{\omega( 2^{n+1}d(z,E))}{ 2^{n+1}d(z,E)  }.
\end{eqnarray*}
Therefore,
\begin{eqnarray} \label{EE2}
\int_{\Gamma} \frac{|f(\zeta)|^2}{|\zeta-z|^2} \, |d\zeta|  &=&
\sum_{n=0}^{\infty} \int_{\Gamma_z^n} \frac{|f(\zeta)|^2}{|\zeta-z|^2} \, |d\zeta| \nonumber \\
&\leq& 2\pi \frac{\omega(2d(z,E))}{1-|z|}+4 \sum_{n=1}^{\infty} \frac{\omega(2^{n+1}d(z,E))}{ 2^{n+1}d(z,E)  } \nonumber\\
&\leq&2\pi \frac{\omega(2d(z,E))}{1-|z|}+ 4\int_{2d(z,E)}^{2}\frac{\omega(t)}{t^2} \, dt.
\end{eqnarray}
On the other hand, on $\T \setminus \Gamma = \{\zeta \in \T : d(\zeta, E) < d(z,E)\}$, we have
\begin{eqnarray}
\int_{\T \setminus \Gamma} \frac{|f(\zeta)|^2}{|\zeta-z|^2} \, |d\zeta| &\leq& \int_{\T \setminus \Gamma}
\frac{\omega( d(\zeta,E))}{|\zeta-z|^2} \, |d\zeta| \nonumber\\ %
&\leq& \omega (d(z,E)) \int_{\T} \frac{|d\zeta|}{|\zeta-z|^2} \nonumber\\ %
&=& 2\pi  \frac{\omega \left( d(z,E) \right)}{1-|z|}. \label{E2}
\end{eqnarray}
Hence, by \eqref{EE2} and \eqref{E2},
\[
\sum_{n=1}^{\infty} (1-|z_n|^2) \, \int_\T \, \frac{|f(\zeta)|^2}{|\zeta-z_n|^2}  \, |d\zeta| \leq 4\pi
 \, \sum_{n=1}^{\infty} \omega (2d(z_n,E)) + 4 (1-|z_n|^2)   \int_{2d(z_n,E)}^{2}\frac{\omega(t)}{t^2}   <\infty.
\]
Let $B$ be the Blaschke product formed by $(z_n)_{n \geq 1}$. It follows from Carleson's formula
\cite{Car60} that
\begin{eqnarray}\label{Carleson-formula}
\cD(Bf) & = & \cD(f) +  \frac{1}{2\pi}\int_\T \,  \sum_{n=1}^\infty  \frac{1-|z_n|^2}{|\zeta-z_n|^2} |f(\zeta)|^2\, |d\zeta| < \infty.
\end{eqnarray}
Thus, $(z_n)_{n \geq 1}$ is a zero set for the Dirichlet space.
\end{proof}

\subsection*{Proof of Theorem \ref{module}}
Since
\begin{eqnarray*}
(1-|z_n|^2) \int_{2d(z_n,E)}\frac{\omega(t)}{t^2} \, dt &\leq& c\; (1-|z_n|^2)+  c\; (1-|z_n|^2) \frac{\omega(2d(z_n,E))}{2d(z_n,E)}\\
&\leq& c\; (1-|z_n|^2)+  c\; \omega(2d(z_n,E))
\end{eqnarray*}
for some constant $c>$ and $(z_n)_{n \geq 1}$ is a Blaschke sequence, the proof follows immediately from Lemma  \ref{T:zeroset-phi}.    \hfill $\square$

\subsection*{Proof of Theorem \ref{capacity}}
Since $\caap (E)=0$, we have $\caap(E_t) \longrightarrow 0$ as $t \longrightarrow 0$. Set
\begin{equation}
\label{eta}
\eta(t)=\psi^{-1}(\caap(E_t)).
\end{equation}
 Then  $\eta(t) \longrightarrow \infty$, as $t \longrightarrow 0$, and
\[
\int_0 \caap(E_t) \, |d\eta^2(t)| \asymp \int_0 \psi(\eta(t)) \, |d\eta^2(t)| \asymp \int^{+\infty}\psi(x) \, xdx< \infty.
\]
Hence, by \cite[Theorem 5.1]{EKR0}, a converse of strong-type estimate for capacity,  there exists a function $\varphi\in \cD$ such that
\[
\Re \varphi(\zeta) \geq \eta (d(\zeta , E)) \quad \text{ and } \quad |\Im \varphi (\zeta)|<\pi /4, \qquad  \text{q.e. on \;\;}  \TT.
\]
Define
$$f =  \exp \big(-\frac{\sqrt{2}}{2} \, e^{\varphi} \big).$$ Note that  by harmonicity $|\Im \varphi(z)|\leq \pi/4$ on $\DD$ and
\begin{eqnarray*}
\int_\DD|f'(z)|^2dA(z)&=&
\frac{1}{2}\int_\DD|\varphi'(z)|^2 e^{2\Re \varphi(z)}e^{-\sqrt{2}\cos(\Im \varphi(z))e^{\Re \varphi(z)}}dA(z)\\
&\leq& \int_\DD|\varphi'(z)|^2 e^{2\Re \varphi(z)}e^{-e^{\Re \varphi(z)}}dA(z)\\
&\leq& \sup_{x\geq 0}( x^2e^{-x}) \int_\DD|\varphi'(z)|^2 \, dA(z).
\end{eqnarray*}
Hence $f \in \cD$. Also, we have

\[
|f(\zeta)|^2\leq \exp \big(-e^{ \eta (d(\zeta , E))} \big) = \omega(d(\zeta,E)), \qquad \zeta \in \TT.
\]
Write
\[
\TT \setminus E=\bigcup_{k \geq 1} I_k,
\]
where $I_k= (e^{i\theta_{2k}},e^{i\theta_{2k+1}})$ are the complementary intervals of the closed set $E$. Note that $\overline{\{e^{i\theta_k}: k \geq 1\}}=E$. Let $(r_{n,k})_{n,k \geq 1}$ be a two-indexed enumeration of $(r_{n})_{n \geq 1}$. Then put
\[
z_{n,k}=r_{n,k} \, e^{i\theta_k}, \qquad n,k \geq 1.
\]
Since on each fixed ray $\theta=\theta_k$, there are infinitely many zeros $z_{n,k}$ which tend to the point $e^{i\theta_k}$, we have  $\overline{\{z_{n,k} : n,k \geq 1\} }\cap \TT=E$. Moreover, $d(z_{n,k},E)=1-r_{n,k}$ and
\begin{eqnarray*}
\frac{\omega((2d(z_{n,k},E))}{1-r_{n,k}}&=& \frac{\omega(2(1-r_{n,k}))}{1-r_{n,k}}\\
&\leq&2\int_{2(1-r_{k,n})}\frac{\omega(t)}{t^2} \, dt\\
&=& 2\int_{2d(z_{n,k},E)}\frac{\omega(t)}{t^2} \, dt.
\end{eqnarray*}
The result now follows from Lemma \ref{T:zeroset-phi}. \hfill $\square$

\begin{rem}
{\rm Note that since $\omega$ is increasing and $\omega(0)=0$, we have
\[
\delta\int_{2\delta}^2\frac{\omega(t)}{t^2}dt=\delta\int_{2\delta}^{\sqrt{\delta}}\frac{\omega(t)}{t^2}dt+\delta\int_{\sqrt{\delta}}^2\frac{\omega(t)}{t^2}dt\leq \omega(\sqrt{\delta})+\sqrt{\delta} \, \omega(2) \longrightarrow 0, \quad \delta \longrightarrow 0.
\]
Hence, we there are $(r_{n})_{n \geq 1}$ which satisfy \eqref{blas}.}
\end{rem}

\section{Some specific examples} \label{S:examp}

Let $E$ be a closed subset of $\TT$ with Lebesgue measure zero, and write $\TT \setminus E = \cup_{n \geq 1} I_n$, where $I_n$'s are the complementary intervals of $E$. Then we have
\begin{equation} \label{E:log-int-00}
\int_\T  \, |\log d(\zeta, E)| \,\, |d\zeta| < \infty \Longleftrightarrow \sum_{n=1}^{\infty} |I_n| \, \log1/|I_n|<\infty.
\end{equation}
See \cite{C}. Such a set is called a {\em Carleson set}. The condition \eqref{E:log-int-00} gives the motivation to consider the closed sets $E$ which satisfy
\begin{equation} \label{E:log-int}
\int_\T  \, |\log \sigma(d(\zeta, E))| \,\, |d\zeta| < \infty,
\end{equation}
where $\omega$ is a positive continuous function. In this situation, the well-known formula
\[
f_{\sigma,E}(z) = \exp \bigg( \int_{\TT} \frac{\zeta+z}{\zeta-z} \,\, \log \sigma(d(\zeta, E)) \,\,\frac{ |d\zeta|}{2\pi} \bigg)
\]
provides an outer function such that whose boundary values $f_{\sigma,E}^*$ satisfy
\[
|f_{\sigma,E}^*(\zeta)| = \sigma(d(\zeta,E)), \qquad \text{a.e. on } \T.
\]
There are several results which show that, under certain conditions on $\omega$ and $E$, this outer function belongs to a specific function space \cite{EKR}. Some relevant cases are treated below for  the generalized Cantor set.

Let us briefly describe the construction of the  generalized Cantor set on $\T$. Put $E_0 = \TT$ and $\ell_0 = 2\pi$. Let $(\lambda_n)_{n \geq 1}$ be a decreasing sequence of positive numbers. We remove an interval of length $\lambda_1$ from the middle of $E_0$. Denote the union of two remaining intervals by $E_1$, and denote the length of each interval in $E_1$ by $\ell_1$. Then we remove two intervals, each of length $\lambda_2$, from the middle of intervals in $E_1$. Let $E_2$ denote the union of the resulting four pairwise disjoint intervals of equal length $\ell_2$. Continuing this procedure, we obtain in the $n$-th  step a compact set $E_n$ which is the union of $2^n$ closed intervals of length equal to $\ell_n$. According to the construction procedure, we must have
\begin{equation} \label{E:rel-ell-lamba}
2\ell_n+\lambda_n=\ell_{n-1}.
\end{equation}
For the reference, denote the intervals of $E_n$ by denoted by $(I_{n,m})_{1\leq m\leq 2^n}$. The compact set $E=\bigcap_{n\geq 1} E_n $ is called the {\em generalized Cantor set}. It is easy to verify that $E$ has Lebesgue measure zero if and only if
\[
\sum_{n=1}^\infty \, 2^{n-1} \, \lambda_n = 2\pi.
\]
The generalized Cantor set has logarithmic capacity zero if and only if
\begin{equation}\label{capacite-cantor}
\sum 2^{-n}\log 1/\ell_n =\infty.
\end{equation}
See \cite[Chapter IV, Theorems 2 and 3]{CL}. Moreover, it becomes a Carleson set if and only if
\begin{equation} \label{E:carl-cant}
\sum_{n=1}^\infty \, 2^n \, \lambda_n \log 1/\lambda_n < \infty.
\end{equation}
The set $E$ is called a {\em perfect symmetric set} with constant ratio $\ell \in (0,1/2)$ if $\ell_n=\ell^n$.   In this case, by \eqref{E:rel-ell-lamba},
\begin{equation} \label{E:form-lam}
\lambda_n= \ell^{n-1}(1-2\ell).
\end{equation}
The classical Cantor set corresponds to $\ell = 1/3$.\\

\noindent {\bf Example 1:} Let $E$ be any Carleson set, and let  $\sigma (t)= t^{\alpha}$ with $\alpha > 1/2$. Then, according to \cite[Corollary 4.2]{EKR}, we have  $f_{\sigma,E}\in \cD$.  Since $d(z,E)\leq (1-|z|)+d(z/|z|,E)$, we deduce the following result from Theorem \ref{module}.

\begin{cor} \label{cor1}
Let $E$ be any Carleson set. Then each Blaschke sequence $(r_ne^{i\theta_n})_{n\geq 1}$ satisfying
\begin{equation} \label{Carlesonzero}			
\sum_{n=1}^\infty \, d(e^{i\theta_n},E)^{2\alpha} < \infty,
\end{equation}
for some exponent $\alpha>1/2$,  is a zero set for $\cD$.
\end{cor}
Let $\mathcal{A}^\infty = \mathcal{C}^\infty(\TT)\cap\Hol(\DD)$. Compare this with \cite[Theorem 3.2]{ms-4} which says, roughly, that if \eqref{Carlesonzero} holds then $E$ is a Carleson set and, moreover, the same result holds if one replaces $\mathcal{D}$ by $\mathcal{A}$.  As a very specific situation of Corollary \ref{cor1}, let $(r_n)_{n \geq 1}$ be any Blaschke sequence, and let $(e^{i\theta_n})_{n \geq 1}$ be any subset of $E$. Put $z_n = r_n e^{i\theta_n}$, $n \geq 1$. In particular, $e^{i\theta_n}$ maybe the end point of a complementary interval of $E$. Then \eqref{Carlesonzero} is clearly satisfied and thus $(r_ne^{i\theta_n})_{n \geq 1}$ is a zero set for $\cD$.\\

\noindent {\bf Example 2:}
Let $\sigma (t)= e^{-1/t^\gamma}$, with $0<\gamma<1$, and let $E$ be a subset set of $\TT$ such that
\begin{equation} \label{E:t-alpha-inty}
\int_{\T} \, \frac{|d\zeta|}{d(\zeta,E)^{\gamma}}<\infty.
\end{equation}
This condition  is also equivalent to $|E|=0$ and
$
\sum_{n=1}^{\infty} |I_n|^{1-\gamma} <\infty,
$
where $(I_n)_{n \geq 1}$ are the complementary intervals of $E$ \cite{EKR0}. In the case of generalized Cantor sets, \eqref{E:t-alpha-inty} becomes equivalent to
\begin{equation}\label{Cantorr}
\sum_{n=1}^\infty \, 2^n \, \lambda_n^{1-\gamma}  < \infty.
\end{equation}
If $E$ is perfect symmetric set with the constant ratio $\ell$, then, by \eqref{E:form-lam}, the condition \eqref{Cantorr} holds if and only if
\[
0 < \gamma < 1+\frac{\log2}{\log \ell}.
\]

Now if   \eqref{E:t-alpha-inty} holds, then $f_{\sigma,E}\in \mathcal{A}^\infty \subset \cD$. The condition \eqref{E:t-alpha-inty} shows that $E$ is a special type of Carleson set. But, for such a sacrifice, we get more freedom for zeros. More precisely, we have the following result.

\begin{cor}\label{cor2}
Let $E$ fulfill \eqref{E:t-alpha-inty}. Then each Blaschke sequence $(z_n)_{n\geq 1}$ satisfying
\begin{equation} \label{E:zero-dist-alpa}
\sum_{n=1}^\infty \, e^{-2/d(z_n,E)^\gamma}  < \infty
\end{equation}
is a zero set for $\cD$.
\end{cor}
The case $E=\{1\}$ of this result was treated in \cite{ms-4}.  As a very specific case of Corollary \ref{cor2}, let $\varepsilon_n=n^{-\frac{1+\gamma}{1-\gamma}}$, and let $E=\{e^{i\varepsilon_n}\}\cup \{1\}$. Then $E$ satisfies  \eqref{E:t-alpha-inty}.  Put
 $z_{n,k}=r_{n,k}e^{i\theta_{n,k}}$ with
\[
 1-r_{n,k}=n^{-n}k^{-k} \quad \mbox{and} \quad \theta_{n,k}= \varepsilon_n+\frac{\varepsilon_{n-1}-\varepsilon_n}{2(\log k)^{2/\gamma}}.
\]
Simple computation shows  that  $(z_{n,k})_{n,k\geq 2}$ is Blaschke sequence satisfying  \eqref{E:zero-dist-alpa}. Hence, $(z_{n,k})_{k,n \geq 2}$ is a zero set for $\cD$. But,  \eqref{Carlesonzero} does not hold and thus this result cannot be deduced from Corollary \ref{cor1}.\\

\noindent {\bf Example 3:}
Let $E$ a be a closed subset of $\TT$ such that
\begin{equation} \label{E:condCar000}
\int_0^2 \frac{ds}{|E_s|}=+\infty,
\end{equation}
where $E_s:=\{\zeta\in \TT \; : \; d(\zeta,E)\leq s\}$. The closed subset $E$ satisfying \eqref{E:condCar000} has logarithmic capacity zero,  but the converse is not true. However, if
$E$ is a generalized Cantor set, then \eqref{E:condCar000} holds if and only if $\caap(E)=0$  \cite[Chapter IV, Theorems 2 and 3]{CL}. Note that this is also equivalent to \eqref{capacite-cantor}. Moreover,
\begin{equation}\label{condcapa}
\caap(E_t)\leq \left(\int_t^2 \frac{ds}{|E_s|} \right)^{-1}.
\end{equation}

\begin{rem}\label{rem}
{\rm Let $\eta(t)=2\log\log\log (1/t)$ and let $E$ be a closed subset of $\TT$ such that
\begin{equation}
\label{E-t-integral}
\int_0 {|E_t|} \, \eta'(t)^2 \, dt<\infty.
\end{equation}
For example, the generalized Cantor set $E$ with $\ell_n=e^{-2^{n}/n^\sigma}$ satisfies this property \cite[Theorem 9.1]{EKR0}, and, by  \eqref{capacite-cantor}, $\caap(E)=0$. Since $E$ satisfies \eqref{E-t-integral}, by \cite[Proposition A.4]{EKR0}, we have
\[
\int_0\caap(E_t) \frac{\log\log\log(1/t)}{t\log(1/t)\log\log(1/t)} dt<\infty.
\]
Hence  $\caap(E_t) \longrightarrow 0$ as $t \longrightarrow 0$, $\caap(E)=0$. Therefore, by \eqref{eta},
\[
\psi^{-1}(\caap (E_t))=\eta(t)=2\log\log\log(1/t)
\]
and
\[
\omega(t)=\exp\big(-e^{\psi^{-1}(\caap (E_t))}\big)\leq \frac{1}{(\log1/t)^2}.
\]
Hence,
\[
\delta\int_\delta\frac{ \omega(t)}{t^2} \, dt \leq \frac{c}{(\log 1/\delta)^2}=o\Big(\frac{1}{\log 1/\delta}\Big), \qquad \delta \longrightarrow 0.
\]
According to Theorem \ref{capacity}, there is a zero set $(z_{n})_{n \geq 1}$ which satisfies
\[
\sum_{n=1}^{\infty} \frac{1}{\log^2 (1-|z_n|)}<\infty\quad\text{ but } \quad \sum_{n=1}^{\infty} \frac{1}{|\log( 1-|z_n|)|}=\infty,
\]
and the set of its accumulation points on $\TT$ is precisely equal to $E$. Therefore, \eqref{blas} holds
 but the zero sequence does not fulfil the Shapiro--Shields condition \eqref{E:suuf-z-ss}.}
\end{rem}

\begin{cor}\label{cor3}
Let $E$ be any closed subset of $\TT$ satisfying \eqref{E:condCar000}. Then each Blaschke sequence $(z_n)_{n\geq 1}$ satisfying
\[
\sum_{n=1}^\infty \exp \left\{ -\left( \int_{2d(z_n,E)}^{2} \frac{ds}{|E_s|}\right)^\alpha \right\} < \infty,
\]
for some $\alpha\in (0,1/2)$, is a zero set for $\cD$.
\end{cor}

\begin{proof}
Since  \eqref{E:condCar000} holds, for each $\beta>0$, we have
\begin{equation}\label{conditionII}
\int_0\frac{ds}{|E_s|(\int_s \frac{dt}{|E_t|})^{1+\beta}}<+\infty.
\end{equation}
Let
\[
\eta_{\alpha,E}(t)=\left(\int_t\frac{ds}{|E_s|}  \right)^{\alpha}
\]
and
\[
\omega(t) = e^{- \eta_{\alpha,E}(t)}.
\]
It clear that $\eta_{\alpha,E}$ is decreasing function and  $\eta_{\alpha,E}(t) \longrightarrow \infty$ as $t \longrightarrow 0$.  Since  $\alpha\in(0,1/2)$, by  \eqref{condcapa} and \eqref{conditionII}, we get
\[
\int_0^2\caap(E_t) \, d\eta_{\alpha,E}^2(t)>-\infty.
\]
Arguing as in the proof of Theorem \ref{capacity}, there exists a function $\varphi\in \cD$ such that
\[
\Re \varphi(\zeta) \geq \eta_{\alpha,E} (d(\zeta , E)) \quad \text{ and } \quad |\Im \varphi (\zeta)|<\pi /4, \qquad  \text{q.e. on \;\;}  \TT.
\]
Define $f =  e^{-\varphi/2}$. We have $f\in \cD$ and $|f(z)|^2\leq \omega(d(z,E))$. Moreover, $\omega$ satisfies \eqref{cond_regularite}. As a matter of fact,  since  $t=O(|E_t|)$ and  $E$ satisfies \eqref{E:condCar000}, integration by parts gives
\begin{eqnarray*}
\int_\delta \frac{\omega(t)}{t^2}dt&=&\frac{\omega(\delta)}{\delta}+ \alpha\int_\delta\frac{1}{|E_t|\big(\int_t \frac{ds}{|E_s|}\big)^{1-\alpha}}\frac{\omega(t)}{t}dt\\
&\leq&\frac{\omega(\delta)}{\delta}+ o\Big( \int_\delta \frac{\omega(t)}{t^2}dt\Big)\\
\end{eqnarray*}
 and the result follows from Theorem \ref{module}.

\end{proof}

\section{Exceptional set of Blaschke products}

Let $B$ be the Blaschke product formed with the sequence $Z=(z_n)_{n\geq 1}$. According to a result of Frostman, $B$ has a finite angular derivative
in the sense of Carath\'eodory at $\zeta\in \TT$ if and only if
\[
\sum_{n=1}^{\infty} \frac{1-|z_n|^2}{|\zeta-z_n|^2}<\infty.
\]
See \cite{Fros-42, Cargo} or \cite[Theorem 1.7.12]{CMR}. Moreover, at such a point,
\[
|B'(\zeta)|= \sum_{n=1}^{\infty} \frac{1-|z_n|^2}{|\zeta-z_n|^2}.
\]
That is why we define the {\em exceptional set}
\[
\mathcal{E}(Z)=\Big\{\zeta\in \TT\text{Ê: } \sum_{n=1}^{\infty} \frac{1-|z_n|^2}{|\zeta-z_n|^2}=\infty\Big\}.
\]
The exceptional set can be quiet large. As a matter of fact, Frostman provided an example so that $\mathcal{E}(Z)=\TT$. However, if $Z=(z_n)_{n\geq 1}$ is a zero set for the Dirichlet space $\cD$, then, by Carleson's formula \eqref{Carleson-formula} and Jensen's inequality, we have
\begin{equation} \label{E:mf-f-0000}
\int_\TT \log \Big(\sum_{n=1}^{\infty} \frac{1-|z_n|^2}{|\zeta-z_n|^2}\Big)|d\zeta|<\infty.
\end{equation}
Hence, the exceptional set $\mathcal{E}(Z)$ must have the Lebesgue measure zero.

Given a measurable function $f$ on $\TT$, the associated distribution function $m_f$ is given  by
\[
m_f(\lambda)=|\{\zeta\in \TT\text{ : } |f(\zeta)|>\lambda\}|,\qquad \lambda\geq 0.
\]
If $ |f|\geq 1$ on $\TT$, then
\begin{equation} \label{E:mf-f}
\int_{\TT} \log|f(\zeta)| \, |d\zeta| = \int_{1}^{\infty} \frac{m_f(\lambda)}{\lambda} \, d\lambda.
\end{equation}
See \cite[Page 347]{m09}. Parallel to this definition and to better study $\mathcal{E}(Z)$, let us define
\[
\mathcal{E}_\lambda(Z)=\left\{\zeta\in \TT\text{ : } \sum_{n=1}^{\infty} \frac{1-|z_n|^2}{|\zeta-z_n|^2 }\geq \lambda \right\}.
\]
Note that
\[
\inf_{\zeta\in \TT} \sum_{n=1}^{\infty} \frac{1-|z_n|^2}{|\zeta-z_n|^2} \geq \frac{1}{4}
 \sum_{n=1}^{\infty} (1-|z_n|^2)= \lambda_0>0.
\]
For simplicity, we assume that $\lambda_0\geq1$. Moreover,
\[
\lim_{\lambda \to +\infty} |\mathcal{E}_\lambda(Z)|= |\mathcal{E}(Z)|,
\]
and thus the latter is zero provided that the limit is equal to zero. However, by \eqref{E:mf-f},
\[
\int_\TT \log \Big(\sum_{n=1}^{\infty} \frac{1-|z_n|^2}{|\zeta-z_n|^2}\Big)|d\zeta|
=
\int_{1}^{\infty }\frac{|\mathcal{E}_\lambda(Z)|}{\lambda} \, d\lambda.
\]
This identity shows that, if \eqref{E:mf-f-0000} holds, then $|\mathcal{E}_\lambda(Z)|$ must tend to zero, although perhaps slowly.

We can say more about the size of this set,  if the zero set $(z_n)_{n \geq 1}$ satisfies the Shapiro-Shields condition  \eqref{E:suuf-z-ss}.

\begin{prop}
If $Z=(z_n)_{n\geq 1}$ satisfies the Shapiro-Shields condition \eqref{E:suuf-z-ss}, then the logarithmic capacity of $\mathcal{E}(Z)$ is zero.
\end{prop}

\begin{proof}
Let $I_{z}$ be the open  arc of $\TT$ with midpoint $z/|z|$ and arclength
\[
|I_{z}|= \bigg( \, (1-|z|) \, \log 1/(1-|z|) \bigg)^{1/2}.
\]
We now verify that, for each $N \geq 1$,
\[
\mathcal{E}(Z) \subset \bigcup_{n\geq N} |I_{z_n}|.
\]
As a matter of fact, $\zeta\notin  |I_{z}|$ implies
\[
\frac{1-|z|^2}{|\zeta-z|^2} \leq\frac{4}{|\log( 1-|z|)|}.
\]
Thus, if $\zeta\notin  |I_{z_n}|$ for all $n\geq N$, then, by  \eqref{E:suuf-z-ss}, we get $\zeta\notin \mathcal{E}(Z)$. Hence,
\begin{eqnarray*}
\caap(\mathcal{E}(Z))&\leq& \caap \Big(\bigcup_{n\geq N} |I_{z_n}|\Big)\\
&\leq &\sum_{n\geq N}
\caap(I_{z_n})\\
&\leq & \sum_{n\geq N}\frac{c_1}{\log 1/|I_{z_n}|}\\
&\leq  &\sum_{n\geq N}\frac{c_2}{|\log(1-|z_n|)|}.
\end{eqnarray*}
with $c_1,c_2$ independent of $N$. Now, let $N \longrightarrow \infty$. By the Shapiro--Shields condition, the last expression tends to zero.
\end{proof}

\begin{prop} \label{p2}
There exists a zero set $(z_n)_{n\geq 1}$  in Dirichlet space $\cD$ such that
\[
\caap(\mathcal{E}(Z))>0.
\]
\end{prop}

\begin{proof}
Let $E$ be a Cantor generalized set on $\TT$. Write $E=\bigcap_{k \geq 1} E_k$, where $E_k=\bigcup_{l=1}^{2^k} I_{k,l}$ with $|I_{k,l}|=\ell_k>0$.  We choose the interval length $\ell_k$ such that
\begin{eqnarray}
\ell = \sup_{k \geq 1} \frac{\ell_{k+1}}{\ell_k} &<& \frac{1}{2}, \label{10}\\
\sum_{k \geq 1} 2^{-k} \log 1/\ell_k &<& \infty,  \label{13}\\
\sum_{k \geq 1} \frac{1}{\log 1/\ell_k} &=& \infty. \label{12}
\end{eqnarray}
For example, if $E$ is the triadic Cantor set, then we can choose $\ell_n=1/3^n$. However, in order to provide more examples we just assume that the sequence $(\ell_k)_{k \geq 1}$ satisfies the above three conditions. Let us see why the above conditions are imposed. In the first place, the condition \eqref{10} implies that
\begin{equation}\label{11}
\sum_{k \geq 1} 2^k \ell_k^2 \leq \sum_{k \geq 1} 2^{k} \ell_k \log 1/\ell_k \leqÊ|\log \ell| \sum_{k \geq 1} k (2\ell)^{k}  <\infty
\end{equation}
and, by \eqref{E:rel-ell-lamba} and \eqref{10}, $\lambda_k \asymp \ell_{k-1}$. Hence, by \eqref{E:carl-cant}, $E$ is a Carleson set. Secondly, the condition \eqref{13} is equivalent to
  $\caap(E)>0$ \cite[\S,  IV Theorem 3]{CL}.

Now we can construct our zero sequence. Let $(z_n)_{n\geq 1}= (z_{k,l})_{l,k\geq 1} $  be such that $z_{k,l}/|z_{k,l}|$ is the center of the interval $I_{k,l}$ and
\[
\bigg( \, (1-|z_{k,l}|) \, \log 1/(1-|z_{k,l}|) \bigg)^{1/2} = \ell_k.
\]
Since $t \leq t \log 1/t$, \eqref{11} guaranties that $(z_n)_{n\geq 1}= (z_{k,l})_{l,k\geq 1} $ is a Blaschke sequence. The condition \eqref{12} ensures that the sequence $(z_{k,l})_{l,k\geq 1} $ does not satisfy the Shapiro--Shields condition \eqref{E:suuf-z-ss} and, moreover,
\[
E=\bigcap_{k \geq 1} E_k\subset  \mathcal{E}(Z).
\]
As a matter of fact, if $\zeta\in E$, then, for each $k \geq 1$, there exists a midpoint $z_{k,l}/|z_{k,l}|$ such $|\zeta-z_{k,l}/|z_{k,l}||\leq \ell_k/2$.  Hence,
\[
\frac{1-|z_{k,l}|^2}{|\zeta-z_{k,l}|^2}\leq\frac{2}{\log 1/\ell_k}.
\]
Therefore, by \eqref{12}, we deduce $\zeta\in  \mathcal{E}(Z)$, and thus $\caap( \mathcal{E}(Z))\geq \caap( E)>0$.  To verify that $(z_{k,l})_{l,k\geq 1} $ is zero set in Dirichlet space, note that since
$z_{k,l}/|z_{k,l}|$ is the center of the interval $I_{k,l}$, we have  $d(z_{k,l}/|z_{k,l}|,E)\leq \ell_k/2$.  Hence, with $\sigma(t)=t$, by \eqref{11}, we have
\[
\sum_{k \geq 1} 2^k \, \sigma^2(2 d(z_{k,l}/|z_{k,l}|,E)) \leq \sum_{k \geq 1} 2^k \ell_k^2 < \infty.
\]
Corollary \ref{cor1} now shows that $(z_n)_{n \geq 1}$ is a zero set for $\cD$.
\end{proof}

In the light of Frostman's Theorem \cite{Fros-42}, see also \cite{Cargo}, Proposition \ref{p2} implies that there exists a Blaschke product which fails to have a finite angular derivative on a set of positive logarithmic capacity.

\vspace{1em}

{\bf Acknowledgments:}  The authors are grateful to the
referee for  his valuable  remarks and suggestions.

\end{document}